\documentclass[preprint,12pt]{elsarticle}
\usepackage{amssymb}
\usepackage{verbatim}
\usepackage{array}
\usepackage{latexsym}
\usepackage{enumerate}
\usepackage{amsmath}
\usepackage{amsfonts}
\usepackage[english]{babel}

\usepackage{amsthm}

\def\K{\mathbb K}

\def\cC{\mathcal C}
\def\cD{\mathcal D}

\def\K{\mathbb{K}}
\def\PG{{\mathrm{PG}}}

\newtheorem{theorem}{Theorem}[section]
\newtheorem{proposition}[theorem]{Proposition}
\newtheorem{lemma}[theorem]{Lemma}
\newtheorem{corollary}[theorem]{Corollary}

\newtheorem{example}[theorem]{Example}

\journal{European Journal of Combinatorics}

\begin{document}

\begin{frontmatter}


\title{Classification of $k$-nets }
\author[focal]{N.~Bogya}
\ead{nbogya@math.u-szeged.hu}
\author[els]{G.~Korchm\'aros\corref{cor1}\fnref{fn1}}
\ead{gabor.korchmaros@unibas.it}
\author[focal,gac]{G.P.~Nagy\fnref{fn2}}
\ead{nagyg@math.u-szeged.hu}
\cortext[cor1]{Corresponding author}
%\cortext[cor2]{Principal corresponding author}
\fntext[fn1]{\,work carried out within the activity of GNSAGA.}
\fntext[fn2]{\,work carried out within the project ``Telemedicine-focused research activities on the field of Matematics, Informatics and Medical sciences''. Project number: TAMOP-4.2.2.A-11/1/KONV-2012-0073.}
\address[focal]{Bolyai Institute, University of Szeged,
Aradi v\'ertan\'uk tere 1, 6725 Szeged (Hungary)}
\address[els]{Dipartimento di
Matematica, Informatica ed Economia, Universit\`a della Basilicata, Contrada Macchia
Romana, 85100 Potenza (Italy)}
\address[gac]{MTA-ELTE Geometric and Algebraic Combinatorics Research Group, P\'azm\'any P. s\'et\'any 1/c, H-1117 Budapest (Hungary)}

\begin{abstract}
 A finite \emph{$k$-net} of order $n$ is an incidence structure consisting of $k\ge 3$ pairwise disjoint classes of lines, each of size $n$, such that every point incident with two lines from distinct classes is incident with exactly one line from each of the $k$ classes. Deleting a line class from a $k$-net, with $k\ge 4$,  gives a \emph{derived} ($k-1$)-net of the same order.
 Finite $k$-nets embedded in a projective plane $\PG(2,\K)$ coordinatized by a field $\K$ of characteristic $0$ only exist for $k=3,4$, see  \cite{knp_k}. In this paper, we investigate $3$-nets embedded in $\PG(2,\K)$ whose line classes are in perspective position with an axis $r$, that is, every point on the line $r$ incident with a line of the net is incident with exactly one line from each class. The problem of determining all such $3$-nets remains open whereas we obtain a complete classification for those coordinatizable by a group. As a corollary, the (unique) $4$-net of order $3$ embedded in $\PG(2,\K)$ turns out to be the only $4$-net embedded in $\PG(2,\K)$ with a derived $3$-net which can be coordinatized by a group. Our results hold true in positive characteristic under the hypothesis that the order of the $k$-net considered is smaller than the characteristic of $\K$.
 \end{abstract}
\end{frontmatter}
\section{Introduction}
%In a projective plane a {\em{$3$-net}} consists of three pairwise disjoint classes of lines such that every point incident with two lines from distinct classes is incident with exactly one line from each of the three classes. If one of the classes has finite size, say $n$, then the other two classes also have size $n$, called the {\em{order}} of the $3$-net.
Finite $3$-nets occur naturally in combinatorics since they are geometric representations of important objects as latin squares, quasigroups, loops and strictly transitive permutation sets. Historically, the concept of $3$-net  arose from
classical differential geometry via the combinatorial abstraction of the concept of a $3$-web; see \cite{NS}. In recent years finite $3$-nets embedded in a projective plane $\PG(2,\K)$ coordinatized by a field $\K$  were investigated in algebraic geometry and resonance theory, see \cite{fy2007,miq,per,ys2004,ys2007}, and a few infinite families of such $3$-nets were constructed and classified, see \cite{knp_3}.

In this paper, we deal with finite $3$-nets embedded in $\PG(2,\K)$ such that the three line classes of the $3$-net appear to be in perspective position with axis $r$, that is, whenever a point $P\in r$ lies on a line of the $3$-net then $P$ lies on exactly one line from each line classes of the $3$-net. If a $3$-net is in perspective position then the corresponding latin square has a transversal, equivalently, at least one of the quasigroups which have the latin square as a multiplicative table has a complete mapping; see \cite[Section 1.4]{DK}. A group has a complete mapping if and only its $2$-subgroups of Sylow are either trivial or not cyclic. This  was conjectured in the 1950's by Hall and Paige \cite{HallPaige}, see \cite[p. 37]{DK}, and proven only recently by Evans \cite{evans}.

As in \cite{knp_3}, most of the known examples in this paper arise naturally in the dual plane of $\PG(2,\K)$, and it is convenient work with the dual concept of a $3$-net embedded in $\PG(2,\K)$.
Formally, a {\em{dual $3$-net}} in $\PG(2,\K)$ consists of a triple $(\Lambda_1,\Lambda_2,\Lambda_3)$ with $\Lambda_1,\Lambda_2,\Lambda_3$ pairwise disjoint point-sets, called {\em{components}}, such that every line meeting two distinct components meets each component in precisely one point. Every component has the same size $n$, the \emph{order} of the dual $3$-net, and each of the $n^2$ lines meeting all components is a \emph{line} of the dual $3$-net.

 A dual $3$-net  $(\Lambda_1,\Lambda_2,\Lambda_3)$ is in \emph{perspective position with a center} $C$, where $C$  is a  point off $\Lambda_1\cup\Lambda_2\cup\Lambda_3$, if  every line through $C$  meeting a component is a line of the dual $3$-net, that is, still meets each component in exactly one point. A dual $3$-net in perspective position has a transversal. Furthermore,
 a dual $3$-net may be in perspective position with different centers although the number of such centers is bounded by the order of the $3$-net. If this bound is attained and every lines through two centers is disjoint from
 $\Lambda_1\cup\Lambda_2\cup\Lambda_3$,
 then the set of the centers can be viewed as a new component $\Lambda_4$ to add to $(\Lambda_1,\Lambda_2,\Lambda_3)$ so that the resulting quadruple $(\Lambda_1,\Lambda_2,\Lambda_3,\Lambda_4)$ is a dual $4$-net, that is, a $4$-net in the dual plane.

 {}From previous work
 \cite[Proposition 3.1]{knp_k}, finite dual $4$-nets have constant cross-ratio, that is,  for every line $\ell$ intersecting the components, the cross-ratio $(P_1,P_2,P_3,P_4)$ with $P_i=\Lambda_i\cap \ell$ is the same. For a dual $3$-net in perspective position with a center $C$, this raises the problem whether for all lines $\ell$ through $C$, the cross-ratio of the points $C, \ell\cap \Lambda_1, \ell\cap \Lambda_2, \ell\cap \Lambda_3$ is the same. By our Proposition  \ref{pr:dual_constcross}, the answer is affirmative. Moreover, in case of more than one centers, the cross-ratio does not depend on which center is referred to. Therefore, any dual $3$-net in perspective position has \emph{constant cross-ratio}.

The problem of classifying all $3$-nets in perspective position remains open and appears to be difficult. Its solution would indeed imply the answer to the main conjecture on finite $4$-nets, namely the non-existence of $4$-nets of order greater than three. Our main result in this context is the following theorem that provides a complete classification for those $3$-nets in perspective position which are coordinatizable by a group.
\begin{theorem} \label{theo1.1.14}
Let $\Lambda$ be a dual $3$-net of order $n$ which is coordinatized by a group. Assume that $\Lambda$ is embedded in a projective plane $\PG(2,\K)$ over an algebraically closed field whose characteristic is either $0$ or bigger than $n$. If $\Lambda$ is in perspective position and $n\neq 8$ then one of the following two cases occur:
\begin{itemize}
\item[\rm(i)] A component of $\Lambda$ lies on a line while the other two lie on a nonsingular conic. More precisely, $\Lambda$ is projectively equivalent to the dual $3$-net given in Lemma \ref{lm:conicline_example}.
\item[\rm(ii)] $\Lambda$ is contained in a nonsingular cubic curve $\cC$ with zero $j(\cC)$-invariant, and $\Lambda$ is in perspective position with at most three centers.
\end{itemize}
\end{theorem}
Theorem \ref{theo1.1.14} provides evidence on the above mentioned conjecture about $4$-nets. In fact, it shows for $n\neq 8$ that the (unique) $4$-net of order $3$ embedded in $\PG(2,\K)$ is the only $4$-net embedded in $\PG(2,\K)$ which has a derived $3$-net coordinatized by a group $G$. This result remains valid in positive characteristic under the hypothesis that that the order $n$ of the $k$-net considered is smaller than the characteristic of $\K$, apart from possibile sporadic cases occurring for $n\in\{12,24,60\}$ and  $G\cong \rm{Alt}_4, \rm{Sym}_4,\,\rm{Alt}_5$, respectively.

The proof of Theorem \ref{theo1.1.14} follows from Propositions \ref{propjan10a}, \ref{propjan12a}, \ref{pr:conicline} and Theorem \ref{thcub} together with the classification of $3$-nets coordinatized by groups, see \cite{knp_3} and \cite{nagypace}, which states that the dual of such a $3$-net is either algebraic (that is, contained in a reducible or irreducible cubic curve), or of  tetrahedron type, or $n=8$ and $G$ is the quaternion group of order $8$. This classification holds true in positive characteristic if the characteristic of $\K$ exceeds the order $n$ of $3$-net and none of the above mentioned special cases for $n=12,24,60$ occurs.

\section{The constant cross-ratio property}
In \cite[Proposition 3.1]{knp_k}, the authors showed that (dual) $4$-nets have constant cross-ratio, that is, for any line intersecting the components, the cross-ratio of the four intersection points is constant. In this section we prove a similar result for  (dual) $3$-nets in perspective positions. Our proof relies on some ideas coming from \cite{knp_k}.

\begin{proposition} \label{pr:pencil_crossratio}
Let $F,G$ be homogeneous polynomials of degree $n$ such that the curves $\mathcal{F}:F=0$ and $\mathcal{G}:G=0$ have $n^2$ different points in common.
%Assume that for any $P\in \mathcal{F} \cap \mathcal{G}$, the tangent lines $t_P(\mathcal{F})$, $t_P(\mathcal{G})$ are different.
Fix nonzero scalars $\alpha, \beta,\alpha',\beta' \in \K$, and define the polynomials
\[H=\alpha F+\beta G, H'=\alpha' F+\beta' G\]
and the corresponding curves $\mathcal{H}:H=0$, $\mathcal{H'}:H'=0$. Then, for all $P\in \mathcal{F} \cap \mathcal{G}$, the tangent lines $t_P(\mathcal{F})$, $t_P(\mathcal{G})$, $t_P(\mathcal{H})$, $t_P(\mathcal{H'})$ have cross-ratio
\[\kappa=\frac{\alpha \beta'}{\alpha'\beta}. \]
\end{proposition}
\begin{proof}
We start with three observations. Notice first that for any $P\in \mathcal{F} \cap \mathcal{G}$, the intersection multiplicity of $\mathcal{F}$ and $\mathcal{G}$ at $P$ must be $1$ by B\'ezout's theorem. This implies that $P$ is a smooth point of both curves, and that the tangent lines $t_P(\mathcal{F})$, $t_P(\mathcal{G})$ are different.
Second, the polynomials $F,G,H,H'$ are defined up to a scalar multiple. Multiplying them by scalars such that the curves $\mathcal{F}, \mathcal{G}, \mathcal{H}, \mathcal{H}'$ don't change, the value of the cross-ratio $\kappa$ remains invariant as well. And third, the change of the projective coordinate system leaves the homogeneous pairs $(\alpha,\beta)$, $(\alpha',\beta')$ invariant, hence it does not affect $\kappa$.

Let us now fix an arbitrary point $P\in \mathcal{F} \cap \mathcal{G}$ and choose the projective coordinate system such that $P=(0,0,1)$, $t_P(\mathcal{F}):X=0$, $t_P(\mathcal{G}):Y=0$. We set $Z=0$ as the line at infinity and switch to affine coordinates $x=X/Z$, $y=Y/Z$. For the polynomials we have
\begin{align*}
F(x,y,1) &=x+f_2(x,y), & G(x,y,1)&=y+g_2(x,y), \\
H(x,y,1) &=\alpha x+ \beta y+h_2(x,y), & H'(x,y,1)&=\alpha' x+\beta'y+ h_2'(x,y),
\end{align*}
with polynomials $f_2,g_2,h_2,h_2'$ of lower degree at least $2$. This shows that the respective tangent lines have equations
\[x=0, y=0, \alpha x+ \beta y=0, \alpha' x+ \beta' y=0,\]
hence, the cross-ratio is indeed $\kappa$.
\end{proof}

Let $(\lambda_1,\lambda_2,\lambda_3)$ be a $3$-net of order $n$, embedded in $\PG(2,\K)$. Let
\[r_1=0, \dots, r_n=0, w_1=0,\ldots, w_n=0, t_1=0,\ldots, t_n=0\]
be the equations of lines of $\lambda_1,\lambda_2,\lambda_3$. Define the polynomials
\[F=r_1\cdots r_n, G=w_1\cdots w_n, H=t_1\cdots t_n;\]
these have degree $n$ and the corresponding curves have exactly $n^2$ points in common. Moreover, the tangents in the intersection points are different; in fact, they are the lines of the dual $3$-net. As explained in \cite{knp_k}, there are scalars $\alpha,\beta \in \K$ such that $H=\alpha F + \beta G$.

Let $\ell$ be a transversal line of $(\lambda_1,\lambda_2,\lambda_3)$, that is, assume that $\ell$ intersects all lines of the $3$-net in the total of $n$ points $P_1,\ldots,P_n$. Let $Q$ be another point of $\ell$, that is, $Q\neq P_i$, $i=1,\ldots,n$. There are unique scalars $\alpha',\beta'$ such that the curve $\mathcal{H'}: \alpha' F+\beta' G=0$ passes through $Q$. As $\mathcal{H}'$ has degree $n$ and $|\ell\cap \mathcal{H'}|\geq n+1$, $\ell$ turns out to be a component of $\mathcal{H}'$. This means that $\mathcal{H}'=\ell \cup \mathcal{H}_0$ for a curve $\mathcal{H}_0$ of degree $n-1$. Moreover, since $\mathcal{H}_0$ cannot pass through $P_1,\ldots,P_n$, the tangent lines of $\mathcal{H}'$ at these points are equal to $\ell$. Proposition \ref{pr:pencil_crossratio} implies the following.

\begin{proposition}[Constant cross-ratio for $3$-nets with transversal] \label{pr:constcross}
Let $\lambda=(\lambda_1,\lambda_2,\lambda_3)$ be a $3$-net of order $n$, embedded in $\PG(2,\K)$. Assume that $\ell$ is a transversal to $\lambda$. Then there is a scalar $\kappa$ such that for all $P\in \ell\cap \lambda$, $m_1\in \lambda_1$, $m_2\in \lambda_2$, $m_3 \in \lambda_3$, $P=m_1\cap m_2\cap m_3$, the cross-ratio of the lines $\ell, m_1, m_2, m_3$ is $\kappa$. \qed
\end{proposition}

The dual formulation of the above result is the following

\begin{proposition}[Constant cross-ratio for dual $3$-nets in perspective position] \label{pr:dual_constcross}
Let $\Lambda=(\Lambda_1,\Lambda_2,\Lambda_3)$ be a dual $3$-net of order $n$, embedded in $\PG(2,\K)$. Assume that $\Lambda$ is in perspective position with respect to the point $T$. Then there is a scalar $\kappa$ such that for all lines $\ell$ through $T$, the cross-ratio of the points $T, \ell\cap \Lambda_1, \ell\cap \Lambda_2, \ell\cap \Lambda_3$ is $\kappa$. \qed
\end{proposition}

In the case when a component of a dual $3$-net is contained in a line, the constant cross-ratio property implies a high level of symmetry of the dual $3$-net. The idea comes from the argument of the proof of \cite[Theorem 5.4]{knp_k}.

\begin{proposition} \label{pr:ca-coll}
Let $\Lambda=(\Lambda_1,\Lambda_2,\Lambda_3)$ be a dual $3$-net embedded in $\PG(2,\K)$. Assume that $\Lambda$ is in perspective position with respect to the point $T$, and, $\Lambda_1$ is contained in a line $\ell$. Then there is a perspectivity $u$ with center $T$ and axis $\ell$ such that $\Lambda_2^u=\Lambda_3$. 
\end{proposition}

\begin{proof}
Let $\kappa$ be the constant cross-ratio of $\Lambda$ w.r.t $T$. Define $u$ as the $(T,\ell)$-perspectivity which maps the point $P$ to $P'$ such that the cross-ratio of the points $T$, $P$, $P'$ and $TP\cap \ell$ is $\kappa$. Then $\Lambda_2^u=\Lambda_3$ holds.
\end{proof}

\section{Triangular and tetrahedron type dual $3$-nets in perspective positions}
\label{TT}
With the terminology used in \cite{knp_3}, a dual $3$-net is \emph{regular} if each of its three components is linear, that is, contained in a line. Also, a regular dual $3$-net is  \emph{triangular} or
of \emph{pencil type} according as the lines containing the components form a triangle or are concurrent.
\begin{proposition}
\label{prop10jan} Any regular dual $3$-net in perspective position is of pencil type.
\end{proposition}
\begin{proof}
Let $(\Lambda_1,\Lambda_2,\Lambda_3)$ be a dual $3$-net in perspective position with center $P$ such that $\Lambda_i$ is contained in a line $\ell_i$ for $i=1,2,3$. Let $R=\ell_1\cap \ell_2$. Take any two points $L_1,L_2 \in \Lambda_1$ and define $M_1=PL_1\cap \Lambda_2$, $M_2=PL_2\cap \Lambda_2$. There exists a perspectivity $\varphi$ with center $R$ and axis $r$ through $P$ which takes $L_1$ to $L_2$ and $M_1$ to $M_2$. For $i=1,2$,  the line $t_i$ through $P,L_i,M_i$ meets $\Lambda_3$ in a point $N_i$. From Proposition \ref{pr:dual_constcross} the cross ratios $(P\, L_1\, M_1\, N_1)$ and $(P\, L_2\, M_2\, N_2)$ coincide. Therefore, $\varphi$ takes $N_1$ to $N_2$ and hence the line $s=N_1N_2$ passes through $R$. If $L_1$ is fixed while $L_2$ ranges over $\Lambda_1$, the point $N_2$ hits each point of $\Lambda_3$. Therefore $\Lambda_3$ is contained in $s$.
\end{proof}
From \cite[Lemma 3]{knp_3}, dual $3$-nets of pencil type do not exist in zero characteristic whereas in positive characteristic they only exist when the order of the dual $3$-net is divisible by the characteristic. This, together with Proposition \ref{prop10jan}, give the following result.
\begin{proposition}
\label{propjan10a} No regular dual $3$-net in perspective position exists in zero characteristic. This holds true for dual $3$-nets in positive characteristic whenever the order of the $3$-net is smaller than the characteristic. \qed
\end{proposition}

Now, let $\Lambda=(\Lambda_1,\Lambda_2,\Lambda_3)$ be tetrahedron type dual $3$-net of order $n\ge 4$ embedded in $\PG(2,\K)$. {}From  \cite[Section 4.4]{knp_3}, $n$ is even, say $n=2m$, and
$\Lambda_i=\Gamma_i\cup \Delta_i$ for $i=1,2,3$ such that each triple
\[\Phi_1=(\Gamma_1,\Gamma_2,\Gamma_3),\, \Phi_2=(\Gamma_1,\Delta_2,\Delta_3), \, \Phi_3=(\Delta_1,\Gamma_2,\Delta_3), \, \Phi_4=(\Delta_1,\Delta_2,\Gamma_3)\]
is a dual $3$-net of triangular type. An easy counting argument shows that each of the $n^2$ lines of $\Lambda$ is a line of (exactly) one of the dual $3$-nets $\Phi_i$ with $1\le i \le 4$.

Now, assume that $\Lambda$ is in perspective position with a center $P$. Then each of the $n$ lines of $\Lambda$ passing through $P$ is also a line of (exactly) one $\Phi_i$ with $1\le i \le 4$. Since $n> 4$, some of these triangular dual $3$-nets, say $\Phi$, has at least two lines through $T$. Let $\ell_1,\ell_2,\ell_3$ denote the lines containing the components of $\Phi$, respectively. Now, the proof of Proposition \ref{prop10jan} remains valid for $\Phi$ showing that $\Phi$ is of pencil type. But this is impossible as we have pointed out after that proof. Therefore, the following result is proven.
\begin{proposition}
\label{propjan12a} No tetrahedron type dual $3$-net in perspective position exists in zero characteristic. This holds true for dual $3$-nets in positive characteristic whenever the order of the $3$-net is smaller than the characteristic. \qed
\end{proposition}
 Propositions \ref{propjan10a} and \ref{propjan12a} have the following corollary.
\begin{corollary}
\label{cor12jan}
Let $\K$ be an algebraically closed field whose characteristic is either zero or greater than $n$. Then no dual $4$-net of order $n$ embedded in $\PG(2,\K)$ has a derived dual $3$-net which is either triangular or of tetrahedron type. \qed
\end{corollary}

%If two fibers are contained in a line then all three and the lines are concurrent. If $n<\chr(K)$ then this is not %possible. Using little more geometry one shows that tetrahedron type dual $3$-nets in perspective positions do not %exist.

\section{Conic-line type dual $3$-nets in perspective positions}
Let $\Lambda=(\Lambda_1,\Lambda_2,\Lambda_3)$ be dual $3$-net of order $n\ge 5$ in perspective position with center $T$, and assume that $\Lambda$ is of conic-line type. Then $\Lambda$ has a component, say $\Lambda_1$, contained in a line $\ell$, while the other two components $\Lambda_2,\Lambda_3$ lie on a nonsingular conic $\cC$. From Proposition \ref{pr:dual_constcross}, $\Lambda$ has constant cross-ratio $\kappa$ with respect to $T$.

%\begin{enumerate}[(i)]
%\item The order $n$ of $\Lambda$ is at least $5$.
%\item $\Lambda$ is embedded in $\PG(2,K)$ where $K$ is an arbitrary field.
%where $K$ is an algebraically closed field of characteristic $p>n$.
%\item $\Lambda$ is in perspective position with center $T$. The  constant cross-ratio of $\Lambda$ with center $T$ is denoted by $\kappa$.
%\item The fiber $\Lambda_1$ is contained in the line $\ell$ and $\Lambda_2\cup \Lambda_3$ is contained in the nonsingular conic $C$.
%\end{enumerate}

%The requirements on $K$ imply that $\ell$ and $C$ intersect in two points.
\begin{lemma}
\label{lemclt}
$\kappa=-1$ and $T$ is the pole of $\ell$ in the polarity arising from $\cC$.
\end{lemma}
\begin{proof}
By Proposition \ref{pr:ca-coll}, there is a $(T,\ell)$-perspectivity $u$ which takes $\Lambda_2$ to $\Lambda_3$. The image $\cC'$ of the conic $\cC$ by $u$ contains $\Lambda_3$ and hence  $\Lambda_3$ is in the intersection of $\cC$ and $\cC'$. Since $\Lambda_3$ has size $n\geq 5$, these nonsingular conics must coincide. Therefore, $u$ preserves $\cC$ and its center $T$ is the pole of its axis $\ell$ is in the polarity arising from  $\cC$. In particular, $u$ is involution and hence $\kappa=-1$.
\end{proof}
Choose our projective coordinate system such that $T=(0,0,1)$, $\ell:Z=0$ and that $\cC$ has equation $XY=Z^2$. In the affine frame of reference, $\ell$ is the line at infinity, $T$ is the origin and $\cC$ is the hyperbola of equation $xy=1$. Doing so, the map $(x,y)\mapsto (-x,-y)$ is the involutorial  perspectivity $u$ introduced in the proof of Proposition \ref{lemclt}.

 As it is shown in \cite[Section 4.3]{knp_3}, $\Lambda$ has a parametrization in terms of an $n^{th}$-root of unity provided that the characteristic of $\K$ is either zero or it exceeds $n$. In our setting
 \[\Lambda_2 = \{(c,c^{-1}), (c\xi,c^{-1}\xi^{-1}), \ldots, (c\xi^{n-1},c^{-1}\xi^{-n+1})\},\]
where $c\in \K^*$ and $\xi$ is an $n$th root of unity in $\K$. Moreover, since $u$ takes $\Lambda_2$ to $\Lambda_3$,
\[\Lambda_3 = \{(-c,-c^{-1}), (-c\xi,-c^{-1}\xi^{-1}), \ldots, (-c\xi^{n-1},-c^{-1}\xi^{-n+1})\}.\]
It should be noted that such sets $\Lambda_2$ and $\Lambda_3$ may coincide, and this occurs if and only if $n$ is even since $\xi^{n/2}=-1$ for $n$ even. Therefore $n$ is odd and then
\[\Lambda_1 = \{(c^{-2}), (c^{-2}\xi), \ldots, (c^{-2}\xi^{n-1}) \}, \]
where $(m)$ denotes the infinite point of the affine line with slope $m$. Therefore the following results are obtained.
\begin{lemma} \label{lm:conicline_example} For $n$ odd, the above conic-line type dual $3$-net is in perspective position with center $T$. \qed
\end{lemma}
\begin{proposition} \label{pr:conicline} Let $\K$ be an algebraically closed field of characteristic zero or greater than $n$. Then every conic-line type dual $3$-net of order $n$ which is embedded in $\PG(2,\K)$ in perspective position is projectively equivalent to the example given in Lemma \ref{lm:conicline_example}. In particular, $n$ is odd.
\end{proposition}
%\begin{proof}
%With this assumptions on $K$, the example is a special case of the one given in \cite[Section 4.3]{knp_3}. %\cite[Proposition 11]{knp_3} implies the uniqueness.
%\end{proof}
%%%%% VALTOZAS
Proposition \ref{pr:conicline} has the following corollary.
\begin{corollary}  Let $\K$ be an algebraically closed field of characteristic zero or greater than $n$. Then no dual $4$-net of order $n$ embedded in $\PG(2,\K)$ has a derived dual $3$-net of conic-line type. 
\end{corollary}
\begin{proof} Since the only $4$-net of order $3$ contains no linear component, and there exist no $4$-nets of order $4$, we may assume that $n\ge 5$. Let $\Lambda_1,\Lambda_2,\Lambda_3,\Lambda_4$ be a dual $4$-net with a derived dual $3$-net  of conic-line type. Without loss of generality, $\Lambda_1$ lies on a line.  Then the derived dual $3$-net $(\Lambda_1,\Lambda_2,\Lambda_3)$ is  of conic-line type. From Proposition \ref{pr:conicline}, there exists a non-singular conic $\cC$ containing $\Lambda_2\cup \Lambda_3$. Similarly, $(\Lambda_1,\Lambda_2,\Lambda_4)$ is  of conic-line type and $\Lambda_2\cup \Lambda_4$ is contained in a non-singular conic $\cD$. Since $\cC$ and $\cD$ share
$\Lambda_2$, the hypothesis $n\ge 5$ yields that $\cC=\cD$. Therefore,  the dual $3$-net $(\Lambda_2,\Lambda_3,\Lambda_4)$  is contained in $\cC$. But this contradicts \cite[Theorem 6.1]{bkm}.
\end{proof}
%%%%VALTOZAS
\section{Proper algebraic dual $3$-nets in perspective positions}
With the terminology introduced in \cite{ys2004}, see also \cite{knp_3}, a dual $3$-net $(\Lambda_1,\Lambda_2,\Lambda_3)$ embedded in $\PG(2,\K)$ is algebraic if its components lie in a plane cubic $\Gamma$ of $\PG(2,\K)$. If the plane cubic is reducible, then  $(\Lambda_1,\Lambda_2,\Lambda_3)$ is either regular or of conic-line type. Both cases have already been considered in the previous sections. Therefore, $\Gamma$ may be assumed irreducible, that is, $(\Lambda_1,\Lambda_2,\Lambda_3)$ is a proper algebraic dual $3$-net.

It is well known that plane cubic curves have quite a different behavior over fields of characteristic $2,3$. Since the relevant case in our work is
in zero characteristic or in positive characteristic greater than the order of the $3$-net considered, we assume that the characteristic of $\K$ is neither $2$ nor $3$. From classical results, $\Gamma$ is either nonsingular or it has at most one singular point, and in the latter case the point is either a node or a cusp. Accordingly, the number of inflection points of $\Gamma$ is $9$, $3$ or $1$. If $\Gamma$ has a cusp then it has an affine equation $Y^2=X^3$ up to a change of the reference frame. Otherwise, $\Gamma$ may be taken in its Legendre form
\[Y^2=X(X-1)(X-c)\] so that $\Gamma$ has an infinite point in $Y_\infty=(0,1,0)$ and it is nonsingular if and only if $c(c-1)\neq 0$.
The projective equivalence class of a nonsingular cubic is uniquely determined by the \textit{$j$-invariant}; see \cite[Section IV.4]{Hartshorne}. Recall that $j$ arises from the cross-ratio of the
four tangents which can be drawn to $\Gamma$ from a point of $\Gamma$ and it takes into account the fact that four lines have six different permutations. Formally, let $t_1,t_2,t_3,t_4$ be indeterminates over $\K$. The cross-ratio of $t_1,t_2,t_3,t_4$ is the rational expression
\[k=k(t_1,t_2,t_3,t_4)=\frac{(t_3-t_1)(t_2-t_4)}{(t_2-t_3)(t_4-t_1)}.\]
$k$ is not symmetric in $t_1,t_2,t_3,t_4$; by permuting the indeterminates, $k$ can take $6$ different values
\[k,\quad \frac{1}{k},\quad 1-k,\quad \frac{1}{1-k},\quad \frac{k}{k-1},\quad 1-\frac{1}{k}.\]
The maps $k\mapsto 1/k$ and $k\mapsto 1-k$ generate the \textit{anharmonic group} of order $6$, which acts regularly on the $6$ values of the cross-ratio. It is straightforward to check that the rational expression
\[u=u(k)=\frac{(k^2-k+1)^3}{(k+1)^2(k-2)^2(2k-1)^2}\]
is invariant under the substitutions $k\mapsto 1/k$ and $k\mapsto 1-k$. Hence, $u$ is a symmetric rational function of $t_1,t_2,t_3,t_4$. Assume that $t_1,t_2,t_3,t_4$ are roots of the quartic polynomial
\[\varphi(t)=\alpha_0+\alpha_1t+\alpha_2t^2+\alpha_3t^3+\alpha_4t^4.\]
Then $u$ can be expressed by the coefficients $\alpha_0,\ldots,\alpha_4$ as
\begin{equation} \label{eq:coeff_expr}
u=\frac{\left( 12\alpha_0\alpha_4-3\alpha_1\alpha_3+ \alpha_2^2 \right)^3}{\left( 72\alpha_0\alpha_2\alpha_4-27\alpha_0\alpha_3^2-27\alpha_1^2 \alpha_4-2\alpha_2^3+9\alpha_1\alpha_2\alpha_3  \right)^2}.
\end{equation}
 It may be observed that if $\Gamma$ is given in Legendre form then its $j$-invariant is
\begin{equation} \label{eq:j}
j(\Gamma)=2^8\frac{(c^2-c+1)^3}{c^2(c-1)^2}.
\end{equation}
Assume that $j(\Gamma)=0$, that is, $c^2-c+1=0$. Then, the Hessian $H$ of $\Gamma$ is
\begin{equation} \label{eq:Hessian}
H:\left(X-\frac{c+1}{3}\right)\left(Y+\sqrt{\frac{2c-1}{3}}\right)\left(Y-\sqrt{\frac{2c-1}{3}}\right)=0.
\end{equation}
This shows that $H$ is the union of three nonconcurrent lines whose intersection points are the \textit{corners} of $\Gamma$.

To our further investigation we need the following result.
\begin{proposition} \label{pr:cubic_j0}
Let $\Gamma$ be an irreducible cubic curve in $\PG(2,\K)$ defined over an algebraically closed field $\K$ of characteristic different from  $2$ and $3$.  For each $i=1,\ldots,7$, take pairwise distinct nonsingular points $P_i,Q_i,R_i\in \Gamma$ such that the triple $\{P_i,Q_i,R_i\}$ is collinear. Assume that there exists a point $T$ off $\Gamma$ such that quadruples $\{T,P_i,Q_i,R_i\}$ are collinear and that their cross-ratio $(T,P_i,Q_i,R_i)$ is a constant $\kappa$. Then $\Gamma$ has $j$-invariant $0$ and $T$ is one of the three corners of $\Gamma$.
\end{proposition}
\begin{proof}
We explicitly present the proof for nonsingular cubics and for cubics with a node, as the cuspidal case can be handled with similar, even much simpler, computation.
Therefore, $\Gamma$ has $9$ or $3$ inflection points. Pick an inflection point off the tangents to $\Gamma$ through $T$. Fix a reference frame such that this inflection point is $Y_\infty=(0,1,0)$ and that $\Gamma$ is in Legendre form
\[Y^2=X(X-1)(X-c).\]
Then $T=(a,b)$ is an affine point and so are $P_i,Q_i,R_i$. Let $\ell$ be the generic line through $T$ with parametric equation $x=a+t,\,y=b+mt$. The parameters of the points of $\ell$ which also lie in $\Gamma$, say $P,Q,R$, are the roots $\tau,\tau',\tau''$ of the cubic polynomial
\[h_1(t)=(a+t)(a+t-1)(a+t-c)-(b+tm)^2.\]
The cross-ratio $k$ of $T,P,Q,R$ is equal to the cross-ratio of $0,\tau,\tau',\tau''$, and its $u(k)$ value can be computed from \eqref{eq:coeff_expr} by substituting
\begin{align*}
\alpha_0&= 0,\\
\alpha_1&= a^3-a^2c-a^2+ac-b^2,\\
\alpha_2&= 3a^2-2a-2ac+c-2bm,\\
\alpha_3&= 3a-1-c-m^2,\\
\alpha_4&= 1.
\end{align*}
{}From this, $u(k)=f(m)^3/g(m)^2$, where $f,g$ have $m$-degree $2$ and $3$, respectively. Let $m_i$ be the slope of the line containing the points $T,P_i,Q_i,R_i$ ($i=1,\ldots,7$). By our assumption,
\[u(\kappa)=\frac{f(m_1)^3}{g(m_1)^2}=\cdots=\frac{f(m_7)^3}{g(m_7)^2},\]
which implies $\frac{f(m)^3}{g(m)^2}=u(\kappa)$ for all $m$, and this holds true even if one of the $m_i$'s is infinite. If the rational function $\frac{f(m)^3}{g(m)^2}$ is constant then either its derivative is constant zero, or $f(m)\equiv 0$, or $g(m)\equiv 0$. In any of these cases,
\begin{equation} \label{eq:3fg2fg}
3f'(m)g(m)-2f(m)g'(m)\equiv 0.
\end{equation}
By setting
\begin{align*}
\beta_0&= 2b(c^2-c+1),\\
\beta_1&= 2ac-2ac^2-2a+3b^2+c^2+c,\\
\beta_2&= -2b(3a-1-c),\\
\beta_3&= 3a^2-2ac+c-2a,
\end{align*}
\eqref{eq:3fg2fg} becomes
\[54(b^2-a(a-1)(a-c))^2 (\beta_0+\beta_1m+\beta_2m^2+\beta_3m^3)\equiv 0.\]
As $T$ is not on $\Gamma$, $b^2-a(a-1)(a-c)\neq 0$ and $\beta_0=\cdots=\beta_3=0$ holds. Define
\begin{align*}
\gamma_0&= -3b(c-2)(2c-1)(c+1),\\
\gamma_1&= -2(c^2-c+1)(6a-4+3c+6ac^2+3c^2-4c^3-6ac),\\
\gamma_2&= -6b(c^2-c+1)^2,\\
\gamma_3&= -8(c^2-c+1)^3.
\end{align*}
Then
\[0=\beta_0\gamma_0+\beta_1\gamma_1+\beta_2\gamma_2+\beta_3\gamma_3=18 c^2(c-1)^2(c^2-c+1).\]
If $c(c-1)=0$ then $\beta_0=2b=0$ and $\beta_1=2(a-c)=0$, which implies $T\in \Gamma$, a contradiction. This proves $c^2-c+1=0$ and $j(\Gamma)=0$ by \eqref{eq:j}. Moreover, $\beta_1=3b^2+c^2+c=0$ amd $\beta_3=\frac{1}{3}(3a-c-1)^2=0$ imply
\[ a=\frac{c+1}{3}, \hspace{1cm} b=\sqrt{\frac{1-2c}{3}},\]
therefore by \eqref{eq:Hessian}, the lines $X=a$ and $Y=b$ through $T$ are components of the Hessian curve of $\Gamma$. This shows that $T$ is a corner point of $\Gamma$.
\end{proof}

\begin{theorem} \label{thm:cubic}
%Let $K$ be an algebraically closed field of characteristic not $2$ or $3$. Let $\Gamma$ be an irreducible cubic curve in $\PG(2,K)$. For $i=1,\ldots,7$, let $P_i,Q_i,R_i$ be different smooth points on $\Gamma$. Let $T$ be a point off $\Gamma$ and assume that for all $i=1,\ldots,7$, the points $T,P_i,Q_i,R_i$ are collinear and their the cross-ratio is a constant $\kappa$. Then the following holds:
Under the hypotheses of Proposition \ref{pr:cubic_j0},
\begin{enumerate}[(i)]
\item the affine reference frame in $\PG(2,\K)$ can be chosen such that $\Gamma$ has affine equation $X^3+Y^3=1$ and that $T=(0,0)$;
\item there is a perspectivity $u$ of order $3$ with center $T$ leaving $\Gamma$ invariant;
\item the constant cross-ratio $\kappa$ is a root of the equation $X^2-X+1=0$.
\end{enumerate}
\end{theorem}
\begin{proof}
(i) According to \cite[Example IV.4.6.2]{Hartshorne}, the plane cubic $\Gamma'$ of equation $X^3+Y^3=Z^3$ has $j$-invariant $0$. Therefore, $\Gamma$ is projectively equivalent to $\Gamma'$. As the Hessian of $\Gamma'$ is $H':XYZ=0$, the corner points of $\Gamma'$ are $(1,0,0)$, $(0,1,0)$ and $(0,0,1)$. The projectivity $(x,y,z)\mapsto (-y,z,x)$, preserves $\Gamma'$, and permutes the corners, thus $T=(0,0,1)$ can be assumed w.l.o.g. 

(ii) $u$ is the map $(x,y)\to (\varepsilon x, \varepsilon y)$, where $\varepsilon$ is a third root of unity in $\K$. (iii) The line $Y=mX$ intersects $\Gamma'$ in the points
\[\left(\frac{1}{\sqrt[3]{1+m^3}},\frac{m}{\sqrt[3]{1+m^3}} \right),
\left(\frac{\varepsilon}{\sqrt[3]{1+m^3}},\frac{\varepsilon m}{\sqrt[3]{1+m^3}} \right),
\left(\frac{\varepsilon^2}{\sqrt[3]{1+m^3}},\frac{\varepsilon^2 m}{\sqrt[3]{1+m^3}} \right),\]
whose cross-ratio with $(0,0)$ is $\kappa=-\varepsilon$.
\end{proof}
Therefore the following result holds.
\begin{theorem}
\label{thcub}
Let $\K$ be an algebraically closed field of characteristic different from $2$ and $3$. Let $\Lambda=(\Lambda_1,\Lambda_2,\Lambda_3)$ be a dual $3$-net of order $n\geq 7$ embedded in $\PG(2,\K)$ which lies on an irreducible cubic curve $\Gamma$. If $\Gamma$ is either singular or is nonsingular with $j(\Gamma)\neq 0$ then $\Lambda$ is not in perspective position. If $j(\Gamma)=0$ then there are at most three points $T_1,T_2,T_3$ such that $\Lambda$ is in perspective position with center $T_i$.
\end{theorem}
\begin{proof}
Let $T$ be a point such that $\Lambda$ is in perspective position with center $T$. By Proposition \ref{pr:dual_constcross}, $\Lambda$ has a constant cross-ratio $\kappa$, hence Theorem \ref{thm:cubic} applies over the algebraic closure of $\K$.
\end{proof}
Theorem \ref{thcub} has the following corollary.
\begin{corollary}  Let $\K$ be an algebraically closed field of characteristic different from $2$ and $3$. Then no dual $4$-net of order $n\geq 7$ embedded in $\PG(2,\K)$ has a derived dual $3$-net lying on a plane cubic. \qed
\end{corollary}
Finally, we show the existence of proper algebraic dual $3$-nets in perspective position.
\begin{example}
Let $\K$ be a field of characteristic different from $2$ and $3$, $\Gamma:X^3+Y^3=Z^3$, $T=(0,0,1)$ and $u:(x,y,z)\mapsto (\varepsilon x,\varepsilon y,z)$ with third root of unity $\varepsilon$. The infinite point $O(1,-1,0)$ is an inflection point of $\Gamma$, left invariant by $u$. Since $u$ leaves $\Gamma$ invariant as well, $u$ induces an automorphism of the abelian group of $(\Gamma,+,O)$; we denote the automorphism by $u$, too. 
%
%Since for any $P\in \Gamma$, the line $TP$ contains $P^u,P^{u^2}$, $P+P^u+P^{u^2}=O$. 
As the line $TP$ contains the points $P^u$, $P^{u^2}$ for any $P\in \Gamma$, one has $P+P^u+P^{u^2}=O$. 
Let $H$ be a subgroup of $(\Gamma,+)$ of finite order $n$ such that $H^{u}=H$. For any $P\in \Gamma$ with $P-P^u \not\in H$, the cosets
\[\Lambda_1=H+P, \hspace{1cm} \Lambda_2=H+P^u, \hspace{1cm} \Lambda_3=H+P^{u^2} \]
form a dual $3$-net which is in perspective position with center $T$. Indeed, for any $A_1+P\in \Lambda_1$ and $A_2+P^u\in \Lambda_2$, the line joining them passes through $-A_1-A_2+P^{u^2}\in \Lambda_3$.
\end{example}

\bibliographystyle{model1a-num-names}
\bibliography{<your-bib-database>}

\end{document}